\documentclass[12pt,a4paper,reqno]{amsart}

\usepackage[utf8]{inputenc}
\usepackage[T1]{fontenc}

\usepackage{hyperref}
\hypersetup{
  colorlinks=true,
  linkcolor=blue,
  filecolor=blue,      
  urlcolor=blue,
  citecolor=blue,
  pdftitle={},
  pdfpagemode=FullScreen,
}

\usepackage{amsfonts}
\usepackage{amsthm}
\usepackage{amsmath}
\usepackage{amscd}
\usepackage[mathscr]{eucal}
\usepackage{indentfirst}
\usepackage{graphicx}
\usepackage{graphics}
\usepackage{pict2e}
\usepackage{epic}
\usepackage{amssymb}

\numberwithin{equation}{section}

\theoremstyle{plain}
\newtheorem{Th}{Theorem}[section]
\newtheorem{Lemma}[Th]{Lemma}
\newtheorem{Cor}[Th]{Corollary}

\newtheorem*{Theorem-non}{Theorem}
\newtheorem*{Theorem-non2}{Theorem}

\theoremstyle{definition}
\newtheorem*{Proof-non}{Proof of Theorem \ref{Maintheorem} assuming Propositions \ref{Prop1} and \ref{Propm}}
\newtheorem*{Proof-non2}{Proof of (1) (\(\mathbf{m_{1}}\)-estimate) in Proposition \ref{Propm} assuming Proposition \ref{Proposition 5.1}}
\newtheorem*{Proof-non3}{Proof of Theorem \ref{Maintheorem2} assuming Propositions \ref{Prop1} and \ref{Propm}}
\newtheorem*{Proof-non4}{Proof of Proposition \ref{Prop1}}
\newtheorem{Def}[Th]{Definition}

\newtheorem{prop}[Th]{Proposition}
\newtheorem{Rem}[Th]{Remark}

\title[Short Exponential Sums and Ternary Correlations]{Short Exponential Sums and Ternary Correlations of Multiplicative Functions}

\author{Jiseong Kim}

\address{Department of Mathematical Sciences, McNeese State University, Lake Charles, LA 70609, USA}
\email{jiseongk51@gmail.com}

\begin{document}

\begin{abstract}
In this paper, we investigate the average behavior of ternary correlations 
for general $k$-divisor-bounded multiplicative functions, assuming certain 
second moment integral bounds for the associated $L$-functions.
Our approach differs from previous methods based on spectral theory or Heath-Brown-type decompositions, and instead combines the circle method with weighted short exponential-sum bounds. 
The key input is short exponential-sum estimates obtained from integral moment bounds for $L$-functions.

\end{abstract}

\subjclass[2020]{}

\maketitle

\tableofcontents

\section{Introduction}

The study of correlations of arithmetic functions is closely connected to many important open problems in number theory, such as the $k$-tuple prime conjecture.
In particular, ternary correlations have been investigated in several works, especially when averaged over shifts.

Averaged versions of ternary or higher-order correlations of arithmetic functions, such as the von Mangoldt function and divisor functions, over shifts $1 \le h \le H$ with $H=X^{\theta+\varepsilon}$ for various values of $\theta$ have been studied in \cite{Mi1991, BP1993, MXTJ2023, MMXTJ2024}.
One of the key ingredients in most of these works is the Heath-Brown identity, which exploits the convolution structure of these functions.

For $1$-bounded multiplicative functions, estimates for ternary correlations at individual shifts are well studied for pretentious multiplicative functions (see, for example, Klurman \cite{K2017} and Darbar \cite{D2017}).
In contrast, for $\mathrm{GL}(n)$ Hecke eigenvalues, results on averaged ternary correlations over shifts have been obtained using spectral theory in automorphic forms.
Note that, except for \cite{Bl2017}, \cite{M2023} and \cite{MMXTJ2024}, most of these works consider ternary correlations where $SL(2,\mathbb{Z})$ Hecke eigenvalues appear as at least one of the factors.

\begin{table}[ht]
\centering
\renewcommand{\arraystretch}{1.15}
\begin{tabular}{@{}p{0.5\textwidth} p{0.15\textwidth} @{}}
\hline
\textbf{Settings (ternary correlations) } & \textbf{References}  \\
\hline
$d_{k}\times d_{2} \times d_{2} $ & \cite{Bl2017} \\
$SL(2,\mathbb{Z})\times SL(2,\mathbb{Z}) \times general $ & \cite{L2018} \\
$SL(2,\mathbb{Z})\times SL(2,\mathbb{Z})\times SL(2,\mathbb{Z})$ & \cite{TH2021} \\
$d_{k} \times d_{\ell} \times$ general, $k,\ell=2,3$ & \cite{M2023} \\
$SL(3,\mathbb{Z})\times SL(2,\mathbb{Z}) \times general $ & \cite{H2024} \\
$SL(4,\mathbb{Z})\times SL(2,\mathbb{Z}) \times general$ & \cite{HG2024} \\
$d_{k} \times d_{\ell} \times d_{m}, k,\ell,m \in \mathbb{N}$ & \cite{MMXTJ2024}
\end{tabular}
\label{tab:ternary-corr-overview}
\end{table}

In this paper, we remove the dependence on convolution structures and spectral methods, allowing us to treat general multiplicative functions using only second-moment bounds for the associated $L$-functions. We establish our results under standard hypotheses, including divisor-boundedness and upper bounds for integral moments.
Our main goal is to study ternary correlations of the form
\begin{equation}\label{tenarystart}
S(X,H)_{f_{1},f_{2},f_{3}}
:=\sum_{|h|\le H}\left(1-\frac{|h|}{H}\right)\sum_{X\le n\le 2X} f_{1}(n)\,f_{2}(n+h)\,f_{3}(n+2h),
\end{equation}
where $f_{1},f_{2},f_{3} \colon\mathbb{N}\to\mathbb{C}$ belong to the class of multiplicative functions defined below.

\begin{Def}
For an integer $k\geq 1$ and a real number $\alpha\ge 0$, the class $\mathcal{F}_k(\alpha)$ consists of multiplicative functions $f\colon\mathbb{N}\to\mathbb{C}$ that satisfy:
\begin{enumerate}

\item \textbf{(analytic condition)} For each $q\in\mathbb{N}$ and each Dirichlet character $\chi$ modulo $q$, define
\[
L(f,\chi,s):=\sum_{n=1}^{\infty}\frac{f(n)\chi(n)}{n^{s}}.
\]
Then $L(f,\chi,s)$ is meromorphic and admits an analytic continuation to $\Re(s)>1/2$.
It is nonzero for $\Re(s)\ge 1$ and has no poles except possibly when $\chi$ is the principal character, in which case it has at most a simple pole at $s=1$.
\item \textbf{(divisor bound)} 
\[
L(f,\chi,s):=\prod_{p} \prod_{j=1}^{k}\left(1-\frac{\alpha_{f,j}(p)\chi(p)}{p^{s}}\right)^{-1}
\]
for some $\alpha_{f,j} \in \mathbb{C}$ such that $|\alpha_{f,j}|=1.$ 
From this, $|f(n)| \leq d_{k}(n)$ where $d_{k}$ is the $k$-fold divisor function. 
\item \textbf{(second moment estimates)} For any $\varepsilon>0$ and real $T\ge 1$,
\begin{equation}\label{lindelof}
\int_{T}^{2T}\left|L(f,\chi,1/2+it)\right|^{2}\,dt
\ll_{\varepsilon} q^{\varepsilon^{2}}(1+T)^{1+\alpha+\varepsilon^{2}}.
\end{equation}
\end{enumerate}
\end{Def}

Also, the class $\mathcal{F}_{k}'(\alpha)$ consists of multiplicative functions $f \in F_{k}(\alpha)$ that their $L$-functions $L(f,\chi,s)$ have no pole at $s=1$, for any Dirichlet character $\chi.$

In general, one expects $\alpha=0$ for most multiplicative functions, such as those in the Selberg class.
\subsection{Short exponential sum estimates}

To study correlations of multiplicative functions, weighted exponential sum estimates have played an important role. 
For example, for $SL(2,\mathbb{Z})$ Hecke-Maass cusp forms, let $\lambda(n)$ denote the normalized Fourier coefficients. It is known that 
(\cite {JW1929})
\[
\sum_{n \le x} \lambda(n)\, e(n\alpha) \ll_{\varepsilon} x^{1/2+\varepsilon}.
\]
For $SL(3,\mathbb{Z})$ Hecke-Maass cusp forms, let $A(n,1)$ denote the normalized Fourier coefficients. It is known that (see \cite{MIL06})
\[
\sum_{n \le x} A(n,1)\, e(n\alpha) \ll_{\varepsilon} x^{3/4+\varepsilon}.
\]
Under the Generalized Riemann Hypothesis, it is proved in \cite{BH1991} that
\[
\sum_{n \le x} \mu(n)\, e(n\alpha) \ll_{\varepsilon} x^{3/4+\varepsilon}.
\]
In the context of short intervals, results are more difficult to obtain. In \cite{T1991} it is shown that
\[
\sum_{x - x^{5/8+\varepsilon} < n \le x} \Lambda(n)\, e(n\alpha) = o\!\left(x^{5/8+\varepsilon}\right),
\]
and, assuming fourth moment integral bounds over short intervals, this is improved to
\[
\sum_{x - x^{3/5+\varepsilon} < n \le x} \Lambda(n)\, e(n\alpha) = o\!\left(x^{3/5+\varepsilon}\right)
\]
where $\Lambda(n)$ is the Von-Mangoldt function.
Later, in \cite{MRT2020}, it is proved that
\[
\int_{X}^{2X} \sup_{\alpha \in [0,1]} \left| \sum_{x-H < n \le x} \Lambda(n)\, e(n\alpha) \right| \, dx = o(XH),
\]
where $H = X^{\varepsilon}$.

These classical results for $\Lambda$ and $\mu$ rely heavily on the Heath--Brown identity to exploit a specific convolution structure.
For general unbounded multiplicative functions that do not admit such a convolution structure with simple components, results of this type are not available.
If one hopes to apply Mellin inversion directly, one would require a uniform bound (in $\alpha$) of the form
\[
\sum_{n=1}^{\infty} \frac{f(n)e(n\alpha)}{n^{1/2+it}} \ll |t|^{\varepsilon},
\]
which is stronger than what is implied by GRH.
Here, we consider weighted exponential sum estimates of a more general nature, which may be of independent interest. 
We prove the following theorem in Section 2.

\begin{Th}\label{Lemma13}
Let $g \in \mathcal{F}_k(0)$, and let $0 \le a < q < X$. 
Let $x \in [X,2X]$, and let $0 < \eta < 1$. 
Then, for any $|\gamma|<1$ with $|\gamma| H^{\eta - \varepsilon/2} \to \infty$, we have
\[
\sum_{m = x}^{x+H} g(m)\, e\!\left(m\left(\frac{a}{q} + \gamma\right)\right)
\ll_{g,k,\varepsilon} (q|\gamma| X)^{1/2+\varepsilon^{2}} H^{1/2} + H^{\eta} (\log X)^{k^{2}-1}.
\]
\end{Th}
\begin{Rem} Let $f_{1} \in \mathcal{F}_{k}'(0).$
We define major arcs $\mathfrak{M}$ and minor arcs $\mathfrak{m}$ as follows. 
Let $\eta = 1 - 7\varepsilon$ and $\beta = H^{-1+8\varepsilon}$. We set
\[
\mathfrak{M} := \bigcup_{1 \le q < Q} \ \bigcup_{\substack{a \bmod q \\ (a,q)=1}} 
\left(\frac{a}{q} - \beta, \frac{a}{q} + \beta\right), 
\qquad 
\mathfrak{m} := [1/Q, 1+1/Q] \setminus \mathfrak{M}.
\]
Under these definitions, for $H \gg X^{1/2+100\varepsilon}$, we obtain the major arc bound
\[
\sup_{\alpha \in \mathfrak{M}} \sum_{x \le n \le x+H} f_{1}(n) e(n\alpha) \ll (qX)^{\varepsilon} q^{1/2} X^{1/2} H^{8\varepsilon},
\]
and the corresponding minor arc estimate
\[
\sup_{\alpha \in \mathfrak{m}} \sum_{x \le n \le x+H} f_{1}(n) e(n\alpha) \ll H^{1/2} \left(\frac{X}{Q}\right)^{1/2+\varepsilon^{2}} + E(x,H),
\]
where $E(x,H)$ denotes a sum over short intervals of length $H^{1-7\varepsilon}$.
Assuming the $k$-divisor-bounded condition, we have the bound
\[
E(x,H) \ll H^{1-\varepsilon}.
\]
Therefore, taking $Q = H^{1/2},$ we have the following uniform bound. \end{Rem} 
\begin{Cor} Let $f_{1} \in \mathcal{F}_{k}'(0),$ and let $ X^{2/3 + 100\varepsilon} \ll H \ll X^{1-\varepsilon}$. Then

\begin{equation}\label{uniformb}
\sup_{\alpha \in [0,1]} \sum_{x \le n \le x+H} f_{1}(n) e(n\alpha) \ll H^{1-\varepsilon}.
\end{equation}

\end{Cor}

\begin{Rem} For the Fourier coefficients $A(1,m)$ of $SL(3,\mathbb{Z})$ Hecke-Maass cusp forms, there are some exponential sum estimates with rational phases. 
For example, it is shown that, under the Ramanujan conjecture, 
\[
\sum_{m \le x} A(1,m) e\!\left(\frac{a m}{q}\right) \ll_{\varepsilon} q^{1/2} x^{2/3+\varepsilon}+qx^{1/3+\varepsilon}
\]
for $1 \leq q \leq  x^{2/3}$
(see \cite{JV2017}).
Most such results are obtained via the Voronoi summation formula, whose applicability depends crucially on the rationality of the phase. 
A direct application of summation by parts does not provide an upper bound better than the uniform bound, unless $q$ is very small.

\end{Rem}
\subsection{Ternary correlations}
Using the above exponential sum estimates, we prove the following. 

\begin{Th}\label{Theorem}
Let $\varepsilon>0$ and $k\ge 1$, and let $f\colon\mathbb{N}\to\mathbb{C}$ be a multiplicative function in the class $\mathcal{F}_k(0)$.
Suppose that $L(f,\chi,s)$ has a simple pole at $s=1$ when $\chi$ is the principal character.
If $X^{10/13+100\varepsilon} \ll H \ll X^{1-\varepsilon}$, then
\[
\sum_{|h|\le H}\left(1-\frac{|h|}{H}\right)\sum_{X\le n\le 2X} f(n)\,f(n+h)\,f(n+2h)
=
XH\sum_{q\ge 1}\sum_{\substack{a \bmod q\\(a,q)=1}} C_q^{3}
+ O\!\left(XH^{1-\varepsilon}\right),
\]
where
\begin{equation}\label{Cq}
C_q := \sum_{q=q_0 q_1}\frac{\mu(q_1)}{\phi(q_1)q_0}\,D_{q_1,q_0},
\end{equation}
for certain constants $D_{q_1,q_0}\ll q^{\varepsilon^{2}}$ (see \eqref{Dq}).
\end{Th}
When $L(f,s)$ has no pole at $s=1$, the following theorem holds, which improves the range of $H.$ 

\begin{Th}\label{Theorem2}
Let $\varepsilon>0$, $\alpha>0$, and $k\ge 1$.
Let $f_{1},f_{2},f_{3}\colon\mathbb{N}\to\mathbb{C}$ be arithmetic functions such that $f_{1}\in\mathcal{F}'_k(0)$, $f_{2},f_{3}$ are $k$-divisor-bounded.
If $X^{2/3+100\varepsilon} \ll H \ll X^{1-\varepsilon}$, then
\[
\sum_{|h|\le H}\left(1-\frac{|h|}{H}\right)\sum_{X\le n\le 2X} f_{1}(n)\,f_{2}(n+h)\,f_{3}(n+2h)
= O\!\left(XH^{1-\varepsilon/2}\right).
\]
\end{Th}
\subsection{Corollaries}

\subsubsection{Non-vanishing over three term arithmetic progressions}

When $L(f,s)$ has a simple pole, the main term in Theorem \ref{Theorem} dominates the error term.
Hence we obtain the following corollary.
\begin{Cor}\label{Cor1}
Let $f\in\mathcal{F}_{k}(0)$, $|f(n)| \le d_{k}(n)$ for all $n,$ and suppose that $L(f,\chi,s)$ has a simple pole at $s=1$ when $\chi$ is a principal character. 
Then there exist infinitely many three-term arithmetic progressions $n,n+h,n+2h$ such that
\[
|f(n)f(n+h)f(n+2h)| \gg_{f} 1.
\]
\end{Cor}

In \cite{TH2021}, the authors showed that there exist infinitely many three-term arithmetic progressions for which
\[
a(n)\,a(n+h)\,a(n+2h)\neq 0,
\]
where $a(n)$ denotes the $n$-th Fourier coefficient of a holomorphic cuspidal Hecke eigenform $F$, under the assumption of Selberg's $1/4$-conjecture.
Note that $|a(n)|^{2}$ is the coefficient sequence of the Rankin-Selberg $L$-function of $F$.
It is thus expected that $|a(n)|^{2}\in\mathcal{F}_{k}(0)$ for some $k$.
So we can improve the result under the assumptions. Here, we put the general result. 

\begin{Cor}\label{Cor2}
Suppose $|a(n)|^{2}\in\mathcal{F}_k(0)$.
Define
\[
B_{c}(X,H):=\bigl|\{(n,h)\in [X,2X]\times[-H,H] : |a(n)a(n+h)a(n+2h)|\ge c\}\bigr|.
\]
Then, for any fixed $c>0$,
\[
\liminf_{X\to\infty}\frac{B_{c}(X,X^{10/13+100\varepsilon})}{X^{1+10/13+100\varepsilon}} \gg_{f,c} 1.
\]
\end{Cor}

\subsection{Sketch of the proof} 
We now give a brief outline of our proof strategy. For real numbers $x$ and $\alpha$, denote
\[
S_f(\alpha;x):=\sum_{x\le n\le x+2H} f(n)e(n\alpha ).
\]
Using the orthogonality relations, we have
\[
\int_{0}^{1} e(n\alpha)\,d\alpha=
\begin{cases}
1 & \text{if } n=0,\\
0 & \text{otherwise.}
\end{cases}
\]
Therefore,
\begin{equation}\begin{split}
\int_{X}^{2X} &\int_{0}^{1} S_{f_{1}}(\alpha;x)S_{f_{3}}(\alpha;x)\,S_{f_{2}}(-2\alpha;x)\,d\alpha\,dx
\\&=\sum_{n,m,r} f_{1}(n)f_{3}(m)f_{2}(r)1_{n+m=2r} \int_{X}^{2X} 1_{x \leq n \leq x+2H }   1_{x \leq r \leq x+2H }  1_{x \leq m \leq x+2H }\; dx. 
\end{split}\end{equation}
When $r=n
+h,$ the above integral is $2H-2|h|$ if $|h| \leq H.$ 
Therefore, we rewrite \eqref{tenarystart} as
\begin{equation}\label{tenaryintegral}
S(X,H)_{f_{1},f_{2},f_{3}}
=\frac{1}{2H}\int_{X}^{2X}\int_{0}^{1}  S_{f_{1}}(\alpha;x)S_{f_{3}}(\alpha;x)\,S_{f_{2}}(-2\alpha;x)\,d\alpha\,dx.
\end{equation}

To prove Theorems \ref{Theorem} and \ref{Theorem2}, our setups of the circle-method are similar except for the size of $q \leq Q.$
For Theorem \ref{Theorem}, let $Q=XH^{-1+5\varepsilon},$ for Thorem \ref{Theorem2}, let  $Q= H^{1/2}.$
Define major and minor arcs
\[
\mathfrak{M}:=\bigcup_{1\le q<Q}\ \bigcup_{\substack{a\bmod q\\(a,q)=1}}\left(\frac{a}{q}-\beta,\frac{a}{q}+\beta\right),
\qquad
\mathfrak{m}:=[1/Q,1+1/Q]\setminus\mathfrak{M},
\]
where $\beta=H^{-1+8\varepsilon}$.
Then
\begin{equation}\begin{split}
&S(X,H)_{f_{1},f_{2},f_{3}} 
\\&=\frac{1}{2H}\int_{X}^{2X}\!\left(\int_{\mathfrak{M}} S_{f_{1}}(\alpha;x) S_{f_{3}}(\alpha;x)S_{f_{2}}(-2\alpha;x)\,d\alpha
+\int_{\mathfrak{m}}S_{f_{1}}(\alpha;x) S_{f_{3}}(\alpha;x)S_{f_{2}}(-2\alpha;x)\,d\alpha \right)dx.
\end{split}\end{equation}
We choose $\beta=H^{-1+8\varepsilon}$ to balance the contributions of the major and minor-arcs.

To prove Theorem \ref{Theorem}, we follow the major arcs treatment in \cite{MRT2019i}.
For the minor arcs we use weighted exponential-sum estimates over short intervals.
Since the dependence on moduli $q\le Q$ is sufficiently mild, the large choice of $Q$ ensures sufficient cancellation on the minor arcs.

To prove Theorem \ref{Theorem2}, we simply use \eqref{uniformb}.

\begin{Rem}
Our arguments are somewhat similar to \cite{LX2018}, which begins with the circle method and then applies weighted exponential-sum bounds.
Note that \cite{LX2018} uses the supremum of a long exponential sum,
\[
\sup_{\alpha\in\mathbb{R}}\left|\sum_{1\le n\le X} f(n)e(\alpha n)\right|\ll X^{\beta}.
\]
Here, we separate major and minor arcs to restrict the range of $\alpha$ in the supremum, and we use second-moment estimates to obtain the needed exponential-sum bounds.

\end{Rem}

\subsection{Notation}
We will use standard notations throughout the paper. For any real number $\alpha$, we write $e(\alpha):=e^{2\pi i\alpha}$.
We denote the Euler totient function by $\phi(n)$, and we write $d_k(n)$ to denote the $k$-fold divisor function.
For integer $q\geq 1$, we write $\chi \pmod q$ for a Dirichlet character modulo $q$.

We set
\begin{equation}\label{Dq}
D_{q_{1},q_{0}}
:=
\operatorname{Res}_{s=1}\left(\sum_{\substack{n=1\\(n,q_1)=1}}^{\infty}\frac{f(q_0 n)}{n^{s}}\right)
\end{equation}
\begin{Rem} A minor technical issue arises when considering the $L$-function with the coefficient $f(q_{0}n)$ instead of $f(n)$. Note that $d_{k}(p^{r})=\frac{(k+r-1)!}{r!(k-1)!}.$
Let $\sigma=\Re(s)\in[1/2,1+\varepsilon]$. 
Using the Euler product and the divisor bound $|f(n)|\leq d_k(n)$, we obtain
\begin{equation}\label{ast}
\begin{split}
&\sum_{\substack{n=1\\(n,q_1)=1}}^{\infty}\frac{f(q_0 n)\chi(n)}{n^{s}}
\\&=L(f,\chi,s)
\prod_{p\mid q}\prod_{j=1}^{k}\left(1-\frac{\alpha_{f,j}(p)\chi(p)}{p^{s}}\right)^{-1}
\prod_{p\mid \frac{q_0}{(q_0,q_1)}}\left(f(q_0)+\frac{f(q_0 p)}{p^{s}}+\cdots\right)
\\
&\ll
d_k\!\left(\frac{q_0}{(q_0,q_1)}\right)\left|L(f,\chi,s)\right|
\prod_{p\mid q}\left(1+\frac{d_{k}(p)+1}{p^{\sigma}}+\frac{d_{k}(p^{2})+1}{p^{2\sigma}}+\cdots\right)
\prod_{p\mid \frac{q_0}{(q_0,q_1)}}\left(1+\frac{O(1)}{p^{\sigma}}+\cdots\right)
\\
&\ll
d_k\!\left(\frac{q_0}{(q_0,q_1)}\right)\left|L(f,\chi,s)\right|
\exp\!\Bigl(O_k\Bigl(\sum_{p\le \log_2 q} p^{-\sigma}\Bigr)\Bigr)
\\
&\ll_{k,\varepsilon}
q^{\varepsilon^{2}}\left|L(f,\chi,s)\right|.
\end{split}
\end{equation}
Here, we use the fact that 
$$1 \leq \left|\prod_{j=1}^{k}\left(1-\frac{|\alpha_{f,j}(p)\chi(p)|}{p^{\sigma}}\right)\right|\left(1+\frac{d_{k}(p)+1}{p^{\sigma}}+\frac{d_{k}(p^{2})+1}{p^{2\sigma}}+\cdots\right)$$
and the convergence of 
$$\left(1+\frac{d_{k}(p)+1}{p^{\sigma}}+\frac{d_{k}(p^{2})+1}{p^{2\sigma}}+\cdots\right).$$  
\end{Rem}

\subsection{Plan of the paper}
In Section \ref{exponential short} we establish an upper bound for weighted exponential sums over short intervals.
The proof of Theorem \ref{Theorem} is given in Section \ref{circle method}.
In Subsection \ref{major} we analyze the major arcs and derive the main term; the minor arcs are treated in Subsection \ref{minor}.
In Section \ref{Sec4} we prove Theorem \ref{Theorem2}.

\section{Exponential Sums over Short Intervals}\label{exponential short}
We now prove the more general version of Theorem \ref{Lemma13}.

\begin{Lemma}\label{Lemma14}
Let $0\leq \alpha<1, g\in \mathcal{F}_{k}(\alpha)$, and let $0\le a<q<X$.
Let $x\in[X,2X]$, and let $0<\eta<1$.
Then, for any $|\gamma|< 1$ with $|\gamma| H^{\eta-\varepsilon/2}\to\infty$, we have
\[
\sum_{m=x}^{x+H} g(m)\, e\!\left(m\left(\frac{a}{q}+\gamma\right)\right)
\ll_{g,k,\varepsilon} (q|\gamma| X)^{(\alpha+1)/2+\varepsilon^{2}}H^{1/2} + E(x,H)
\]
\end{Lemma}
where $$E(x,H)\ll \sum_{x-H^{\eta}\le m\le x} d_k(m)
+\sum_{x+H\le m\le x+H+H^{\eta}} d_k(m).$$
\begin{proof}
Let
\[
A(s):=\sum_{n=1}^{\infty} g(n)e\!\left(\frac{an}{q}\right)n^{-s}.
\]
For simplicity, assume $(a,q)=1$.
It is known that when $(n,q)=1$, we have
\[
e\!\left(\frac{an}{q}\right)=\frac{1}{\phi(q)}\sum_{\chi \bmod q}\tau(\bar{\chi})\,\chi(an),
\]
where
\[
\tau(\chi):=\sum_{\substack{1\le m\le q\\ (m,q)=1}}\chi(m)\,e\!\left(\frac{m}{q}\right)
\]
is the Gauss sum.
Therefore,
\begin{equation}\label{additive}
\sum_{n=1}^{\infty} g(n)e\!\left(\frac{an}{q}\right)n^{-s}
=
\sum_{q=q_0 q_1}\frac{1}{\phi(q_1) q_0^{s}}
\sum_{\chi \bmod q_1}\tau(\bar{\chi})\,\chi(a)
\sum_{\substack{n=1\\ (n,q_1)=1}}^{\infty} g(q_0 n)\chi(n)\,n^{-s}.
\end{equation}

Let $\psi$ be a smooth compactly supported function with $\operatorname{supp}(\psi)\subset [-H^{-1+\eta},\,1+H^{-1+\eta}]$, such that $\psi(x)=1$ for $0\le x\le 1$, $\psi(x)\le 1$ elsewhere, and
\[
\psi^{(j)}(x)\ll_{j} H^{(1-\eta)j}.
\]
Define
\[
B(s,\gamma):=\int_{0}^{\infty} \psi\!\left(\frac{y-x}{H}\right)e(\gamma y)\,y^{s-1}\,dy.
\]
By Mellin transform and inversion,
$$\mathcal{M} \left(\psi\left(\frac{y-x}{H}\right)e(\gamma y)\right) (s):=  \int_{0}^{\infty} \psi\left(\frac{y-x}{H}\right)e(\gamma y) y^{s=1}dy, $$
$$\psi\left(\frac{m-x}{H}\right)e(m\gamma )= \frac{1}{2\pi i} \int_{(\sigma)}\mathcal{M} \left(\psi\left(\frac{y-x}{H}\right)e(\gamma y)\right)(s) \; m^{-s}ds.$$
Therefore, 
\[
\sum_{m=1}^{\infty} g(m)e\!\left(m\left(\frac{a}{q}+\gamma\right)\right)\psi\!\left(\frac{m-x}{H}\right)
=
\frac{1}{2\pi i}\int_{(\sigma')} A(s)\,B(s,\gamma)\,ds,
\]
where $\sigma'=1+\frac{1}{\log X}$.
Note that 
after removing the weight $\psi$, we have
\begin{equation}
\begin{split}
\sum_{m=1}^{\infty} & g(m)e\!\left(m\left(\frac{a}{q}+\gamma\right)\right)\psi\!\left(\frac{m-x}{H}\right)
-\sum_{m=x}^{x+H} g(m)e\!\left(m\left(\frac{a}{q}+\gamma\right)\right)
\\
&\ll_{\varepsilon}
\sum_{x-H^{\eta}\le m\le x} d_k(m)
+\sum_{x+H\le m\le x+H+H^{\eta}} d_k(m).
\end{split}
\end{equation}
Define
\[
L(g,\chi,s;q_0,q_1):=\sum_{\substack{n=1\\ (n,q_1)=1}}^{\infty} g(q_0 n)\chi(n)\,n^{-s}.
\]
Using \eqref{additive}, we get
\begin{equation}\label{absa}
\begin{split}
\int_{(\sigma)} A(s)B(s,\gamma)\,ds
&=
\sum_{q=q_0 q_1}\sum_{\chi \bmod q_1}\tau(\bar{\chi})\,\chi(a)\,\frac{1}{\phi(q_1)}
\int_{(\sigma)}L(g,\chi,s;q_0,q_1) q_{0}^{-s}\,B(s,\gamma)\,ds
\\
&\quad
+O_{l}\!\left(Hd_{2}(q) \frac{1}{\left(|\gamma| H^{\eta}\right)^{l}}\right),
\end{split}
\end{equation}
where $1/2\le \sigma<1$, for any $l>0.$
The error term comes from the possible simple pole of $L(g,\chi_{q_1,0},s;q_0,q_1)$ together with the bound $B(1,\gamma)\ll_{l} H|H^{\eta}\gamma|^{-l}$, so it is bounded by 
$$ \sum_{q=q_{0}q_{1}} \frac{|\mu(q_{1}|}{\phi(q_{1})q_{0}} \textrm{Res}_{s=1}L(g,\chi_{q_{1},0},s; q_{0},q_{1}) B(1,\gamma) = O_{l}\!\left(Hd_{2}(q) \frac{1}{\left(|\gamma| H^{\eta}\right)^{l}}\right)$$ for any $l>0.$ Since $|\gamma|H^{\eta}=o(H^{\epsilon/2}),$
taking $l$ sufficiently large makes this negligible.

We split the $t$-integral into
\[
\int_{|t|\ge 5\pi |\gamma| X} + \int_{|t|\le 5\pi |\gamma| X}.
\]

For the first integral, we use \cite[Lemma 8.1]{BKY13}. 
Note that $\left(\psi\!\left(\frac{x}{H}\right)\right)^{(j)}\ll H^{-j\eta}$, $h(y):=t\log y+2\pi \gamma y$ satisfies
$|\frac{d  (h(y))}{dy} |= |\frac{t}{y} + 2\pi \gamma| \geq |\gamma| x,$
$|\frac{d^{k}  h}{dy^{k}}| \ll \frac{|t|}{|y|^{k}}$ for $k \geq 2.$
Under the assumption $\gamma H^{\eta-\varepsilon/2}\to\infty$, \cite[Lemma 8.1]{BKY13} yields, for any $A>0$,
\[
B(s,\gamma)\ll_A  H \min\!\left(\frac{X\cdot \frac{|t|}{X}}{\sqrt{|t|}}, \frac{|t|}{X}H^{\eta}\right)^{-A}.
\]
Hence this range is negligible.

Now consider the second integral. We estimate
\begin{equation}
\begin{split}
\int_{|t|\leq 5\pi |\gamma| X} |B(\sigma+it,\gamma)|^2\,dt
&=
\int_{|t|\leq 5\pi |\gamma| X}\int_0^{\infty}\!\!\int_0^{\infty}
\psi\!\left(\frac{y-x}{H}\right)\psi\!\left(\frac{z-x}{H}\right)
\\&\qquad\qquad\cdot e\bigl(\gamma(y-z)\bigr)\left(\frac{y}{z}\right)^{it}(yz)^{\sigma-1}\,dy\,dz\,dt.
\end{split}
\end{equation}
Interchanging the order of integration yields
\[
\int_{|t|\leq 5\pi |\gamma| X} |B(\sigma+it,\gamma)|^2\,dt
\ll
X^{2\sigma-2}
\int_{x-H^{\eta}}^{x+H+H^{\eta}}\!\!\int_{x-H^{\eta}}^{x+H+H^{\eta}}
\min\!\left(|\gamma|X,\frac{1}{|\log(y/z)|}\right)\,dy\,dz.
\]
When $|y-z|<1/|\gamma|$, we bound the minimum by $|\gamma|X$:
\[
\int_{x-H^{\eta}}^{x+H+H^{\eta}}
\left(\int_{z-\frac{1}{|\gamma|}}^{z+\frac{1}{|\gamma|}} |\gamma|X\,dy\right)\,dz
\ll HX.
\]
Moreover, since $z \leq 2y$, 
\begin{equation}
\begin{split}
&\int_{x-H^{\eta}}^{x+H+H^{\eta}}
\left(\int_{x-H^{\eta}}^{z-\frac{1}{|\gamma|}}+\int_{z+\frac{1}{|\gamma|}}^{x+H+H^{\eta}}\right)
\frac{1}{|\log(y/z)|}\,dy\,dz
\\
&\qquad\ll
\int_{x-H^{\eta}}^{x+H+H^{\eta}}
\left(\int \frac{y}{|z-y|}\,dy\right)\,dz
\ll HX\log X.
\end{split}
\end{equation}
Therefore,
\[
\int_{|t|\leq 5\pi |\gamma| X} |B(\sigma+it,\gamma)|^2\,dt
\ll X^{2\sigma-1}H\log X.
\]

Finally, by H\"older's inequality and \eqref{lindelof},
\begin{equation}
\begin{split}
\int_{|t|\leq 5\pi |\gamma| X}
&\Bigl|L(g,\chi,\sigma+it;q_0,q_1)\,B(\sigma+it,\gamma)\Bigr|\,dt
\\
&\ll
\left(\int_{|t|\leq 5\pi |\gamma| X}\bigl|L(g,\chi,\sigma+it;q_0,q_1)\bigr|^2\,dt\right)^{1/2}
\left(\int_{|t|\leq 5\pi |\gamma| X}|B(\sigma+it,\gamma)|^2\,dt\right)^{1/2}
\\
&\ll
(|\gamma|X)^{1/2+\alpha/2+\varepsilon^2/2} \,q^{\varepsilon^{2}/2}\,X^{\sigma-1/2}H^{1/2}\,(\log X)^{1/2}.
\end{split}
\end{equation}

Thus,
\[
\int_{(\sigma)} A(s)B(s,\gamma)\,ds
\ll_{g,k,\varepsilon}
\sum_{q=q_0 q_1}\ \sum_{\chi \bmod q_1}
|\tau(\bar{\chi})|\,\frac{1}{\phi(q_1)}
\frac{q^{\varepsilon^{2}}}{q_{0}^{\sigma}}
\Bigl(X^{\sigma+\alpha/2+\varepsilon^2/2}H^{1/2}|\gamma|^{\alpha/2+1/2+\varepsilon^2/2}\Bigr)(\log X)^{1/2}.
\]
Taking $\sigma=1/2$ and using $|\tau(\bar{\chi})|\ll q_1^{1/2}$, this is bounded by
\[
q^{1/2+\varepsilon^2}X^{1/2+\alpha/2+\varepsilon^2/2}H^{1/2}|\gamma|^{1/2+\alpha/2+\varepsilon^2/2}(\log X)^{1/2}\sum_{q=q_{0}q_{1}} \frac{1}{q_{0}^{1/2}}.
\]

This completes the proof.
\end{proof}
Therefore, if $a(n)\in\mathcal{F}_{k}(0)$, then
\[
\frac{1}{X}\int_{X}^{2X}\sup_{\alpha\in\mathfrak{m}}\bigl|S_a(-2\alpha;x)\bigr|dx
\ll_{g,k,\varepsilon}
q^{1/2+\varepsilon^{2}}X^{1/2+\varepsilon^{2}}H^{1/2}\left|\alpha-\frac{a}{q}\right|^{1/2+\varepsilon^{2}}
+H^{\eta} (\log X)^{k-1},
\]
as long as
\[
\left(\alpha-\frac{a}{q}\right)H^{\eta-\varepsilon/2}\to\infty.
\]
Later, to apply the above bound to achieve a saving of $H^{\varepsilon}$, we require
\[
H^{-1+7.5\varepsilon}=o\left(|\gamma|\right),
\qquad\text{and}\qquad
\frac{X}{Q}\ll H^{1-2\varepsilon} \quad \text{when} \quad Q \neq 1.
\]

\begin{Lemma}\label{moving2}
Assume that for any sufficiently large $T,$
\[
\int_{-T}^{T}\bigl|L(f,\chi,1/2+it)\bigr|^{2}\,dt \ll q^{\varepsilon}(T+1)^{1+\alpha+\varepsilon}
\]
for some $\alpha>0$.
Then for any sufficiently large $T$, there exists $T_0\in[T,2T]$ such that
\[
\sup_{\sigma\in[1/2,1]}\bigl|L(f,\sigma+iT_0)\bigr|
\ll_{\varepsilon} q^{\varepsilon/2}T^{\alpha/2+\varepsilon}.
\]
\end{Lemma}
\begin{proof}
The proof follows in a standard way, see, for example, \cite[Lemma 2]{RS1991}, or apply the Phragm\'en-Lindel\"of principle.
\end{proof}

\section{Proof of Theorem \ref{Theorem}}\label{circle method}
\subsection{Major arcs}\label{major}
We first represent the exponential sum as a combination of averages of $f_{1}(n)$ twisted by Dirichlet characters.
As usual, the main term arises from the contribution of principal characters.
We then apply Lemma \ref{moving2} to control the error term. Note that $\beta=H^{-1+8\varepsilon}.$
To obtain precise asymptotics, we use the Poisson summation formula.

\begin{Lemma}[Summation by parts]\label{sum}
Let $x\in[X,2X]$. Then for any $\gamma\in\mathbb{R}$,
\[
\sum_{x \le n \le x+2H} f(n)\, e\!\left(\frac{an}{q} + \gamma n\right)
= \int_{x}^{x+2H} C_q\, e(\gamma y)\,dy
+ O\!\left(E(x+2H)+\int_{x}^{x+2H} |\gamma|\,E(x'')\,dx''\right),
\]
where
\[
E(x''):=\left|\sum_{x \le n \le x''} f(n)\, e\!\left(\tfrac{an}{q}\right) - \int_{x}^{x''} C_q\,dx\right|,
\]
and $C_q$ is defined in \eqref{Cq}.
\end{Lemma}

\begin{proof}
See \cite[Lemma 2.1]{MRT2019i}.
\end{proof}
\begin{Lemma}\label{attachingrationalerror}
Let $x'\in[x,x+2H]$. Then
\[
E(x') \ll_{\varepsilon} q^{1/2+2\varepsilon^{2}}\,X^{1/2+2\varepsilon}.
\]
\end{Lemma}

\begin{proof}
First, we use the identity
\[
\sum_{n=1}^{\infty} f(n)\, e\!\left(\frac{an}{q}\right)n^{-s}
=
\sum_{q=q_0q_1}\frac{1}{\phi(q_1)\,q_0^{s}}
\sum_{\chi \bmod q_1}\tau(\overline{\chi})\,\chi(a)
\sum_{\substack{n=1\\(n,q_1)=1}}^{\infty} f(q_0n)\chi(n)\,n^{-s}.
\]

Let
\[
L(\chi,s):=\sum_{n=1}^{\infty} f(n)\chi(n)\,n^{-s},
\] 
$\chi_{q_1,0}$ denote the principal character modulo $q_1$,
and let $L^{\ast}(\chi,s)$ denote the Dirichlet series modified to reflect the condition $(n,q_1)=1$ and the substitution $n\mapsto q_0n$.
Using \eqref{ast}, we have
\[
\bigl|L(\chi,s)^{-1}L^{\ast}(\chi,s)\bigr|
\ll_{\varepsilon} q^{\varepsilon^{2}}
\qquad \text{for } \Re(s)\in[1/2,1+\varepsilon].
\]

Applying Perron's formula, we obtain, for any $0<T\ll_{\varepsilon} X^{1-\varepsilon}$,
\begin{equation}\label{perron-attach}
\begin{split}
\sum_{x \le n \le x'} f(n)\, e\!\left(\tfrac{an}{q}\right)
&=
\sum_{q=q_0q_1}\frac{1}{\phi(q_1)}
\sum_{\chi \bmod q_1}\tau(\overline{\chi})\,\chi(a)\,
\frac{1}{2\pi i}
\int_{1+\varepsilon-iT}^{1+\varepsilon+iT}
L^{\ast}(\chi,s)\,\frac{(x')^{s}-x^{s}}{q_0^{s}s}\,ds \\
&\quad + O_{\varepsilon}\!\left(\sum_{q=q_{1}q_{0}} \frac{q_{1}^{3/2}}{\phi(q_{0})}\left(\frac{X}{q_{1}}\right)^{1+2\varepsilon}T^{-1}\bigr)\right).
\end{split}
\end{equation}
Using \eqref{lindelof} together with Lemma \ref{moving2}
(for convenience, let \(T\) denote the parameter $T_{0}$ in Lemma \ref{moving2}),
 we shift the line of integration and obtain
\[
\int_{1+\varepsilon-iT}^{1+\varepsilon+iT}
L^{\ast}(\chi,s)\,\frac{(x')^{s}-x^{s}}{q_0^{s}s}\,ds
=
\operatorname{Res}_{s=1} L^{\ast}(\chi,s)\,\frac{x'-x}{q_0}
+O\!\bigl(\left|F_{\chi,q_0,q_1,T}(x')\right|+\left(\frac{X}{q_{0}}\right)^{1+\varepsilon}T^{-1+\varepsilon}q^{\varepsilon^{2}}\bigr),
\]
where
\[
F_{\chi,q_0,q_1,T}(x')
:=
\int_{1/2-iT}^{1/2+iT}
L^{\ast}(\chi,s)\,\frac{(x')^{s}-x^{s}}{q_0^{s}s}\,ds.
\]

Moreover, for any $T'\gg 1$,
\[
\int_{1/2+iT'}^{1/2+2iT'}
L^{\ast}(\chi,s)\,\frac{(x')^{s}-x^{s}}{q_0^{s}s}\,ds
\ll \left(\frac{x}{q_0}\right)^{1/2}(qT')^{\varepsilon^{2}}.
\]
Using a dyadic decomposition with $T=X^{1/2}$, we obtain
\[
\sum_{q=q_0q_1}\ \sum_{\substack{\chi \bmod q_1\\ \chi\neq \chi_{q_1,0}}}
\frac{\tau(\overline{\chi})\chi(a)}{\phi(q_1)}
\left(\int_{1/2-iT}^{1/2+iT}
L^{\ast}(\chi,s)\,\frac{(x')^{s}-x^{s}}{q_0^{s}s}\,ds
+O_{\varepsilon}\!\left(q^{\varepsilon^{2}}X^{1+\varepsilon}T^{\varepsilon-1}\right)\right)
\ll_{\varepsilon} q^{1/2+2\varepsilon^{2}}\,X^{1/2+2\varepsilon}.
\]

Finally, since $L^{\ast}(\chi_{q_1,0},s)$ has a simple pole at $s=1$, we have
\[
\operatorname{Res}_{s=1}L^{\ast}(\chi_{q_1,0},s)=D_{q_1,q_0}
\qquad\text{(see \eqref{Dq})}.
\]
Hence, the contribution of the principal characters equals $C_q\,(x'-x)$.
This yields the desired bound for $E(x')$.
\end{proof}

Combining the preceding lemmas yields the following.

\begin{prop}\label{majorprop}
Let $x\in[X,2X]$ and $|\gamma|\ \leq \beta$. Then
\[
\sum_{x \le n \le x+2H} f(n)\, e\!\left(\frac{an}{q} + \gamma n\right)
=
\int_{x}^{x+2H} C_q\, e(\gamma y)\,dy
+ O_{\varepsilon}\!\left(H^{8\varepsilon}\,q^{1/2+2\varepsilon^{2}}\,X^{1/2+2\varepsilon}\right).
\]
\end{prop}
Note that 
\[
\sum_{q=q_0q_1}\frac{\mu(q_1)}{\phi(q_1)\,q_0}\,
\operatorname{Res}_{s=1}L^{\ast}(\chi_{q_1,0},s)
\ll
\sum_{q=q_0q_1}\left|\frac{\mu(q_1)}{\phi(q_1)\,q_0}\,
\operatorname{Res}_{s=1}L^{\ast}(\chi_{q_1,0},s)\right|
\ll q^{2\varepsilon^{2}-1}.
\]
By Proposition \ref{majorprop}, the major-arc contribution equals
\begin{equation}\label{totalintegral}
\begin{split}
\frac{1}{2H} \sum_{q < Q} \sum_{\substack{a \bmod q \\ (a,q)=1}}
\int_{|\gamma| < \beta} \int_{X}^{2X} &
\left( \int_{x}^{x+2H} C_q\, e(\gamma y)\, dy
+ O_{\varepsilon}\!\left(H^{8\varepsilon} q^{1/2 + 2\varepsilon^{2}} X^{1/2 + 2\varepsilon}\right) \right)^{2} \\
&\times \left( \int_{x}^{x+2H} C_q\, e(-2\gamma y)\, dy
+ O_{\varepsilon}\!\left(H^{8\varepsilon} q^{1/2 + 2\varepsilon^{2}} X^{1/2 + 2\varepsilon}\right) \right)
\, dx\, d\gamma .
\end{split}
\end{equation}

Since
\[
\int_{x}^{x+2H} C_q\, e(\gamma y)\, dy \ll q^{2\varepsilon^{2}-1}\,H,
\]
the total contribution of the error terms in \eqref{totalintegral} is
\[
\ll X^{3/2+2\varepsilon}Q^{1/2+6\varepsilon^{2}}\,H^{16\varepsilon}
\;+\;
Q^{7/2+6\varepsilon^{2}}\,X^{5/2+6\varepsilon}\,H^{-2+32\varepsilon}.
\]
With $Q=XH^{-1+5\varepsilon}$ and $H\gg X^{10/13+100\varepsilon}$, the above contribution is bounded by 
\[
\ll_{\varepsilon} XH^{1-\varepsilon}.
\]

\subsection{The main term}
We now extract the asymptotic from
\[
\frac{1}{2H} \sum_{q < Q} \sum_{\substack{a \bmod q \\ (a,q)=1}}
\int_{|\gamma| < \beta} \int_{X}^{2X}
\left( \int_{x}^{x+2H} C_q\, e(\gamma y)\, dy \right)^{2}
\left( \int_{x}^{x+2H} C_q\, e(-2\gamma y)\, dy \right)
\, dx\, d\gamma .
\]

The main term becomes
\[
\frac{1}{2H} \sum_{q < Q} \sum_{\substack{a \bmod q \\ (a,q)=1}} C_q^3
\left(\int_{\gamma\in\mathbb{R}}-\int_{|\gamma|>\beta}\right)
\left(
\int_{X}^{2X}
\left(\int_{x}^{x+2H} e(\gamma y)\,dy\right)^{2}
\left(\int_{x}^{x+2H} e(-2\gamma y)\,dy\right)
\,dx
\right)\,d\gamma.
\]

Note that
\[
\sum_{q>Q}\ \sum_{\substack{a \bmod q\\ (a,q)=1}} C_q^{3}\ll 1.
\]
Using the trivial bound $\int_{x}^{x+H} e(\gamma y)\,dy\ll |\gamma|^{-1}$, the tail integral
$\int_{|\gamma|>\beta}\cdots\,d\gamma$ contributes at most $XH^{1-16\varepsilon}$.

Let $1_{[a,b]}$ denote the indicator of $[a,b]$, and denote the Fourier transform of a function $b(t)$ as
\[
\widehat{b}(\zeta):=\int_{\mathbb{R}} b(t)\,e(-\zeta t)\,dt.
\]
Then the first integral equals
\begin{equation}\label{Poissonset}
\int_{X}^{2X}\int_{\mathbb{R}}
\Bigl(\widehat{1_{[x,x+2H]}}(-\gamma)\Bigr)^{2}\,
\widehat{1_{[x,x+2H]}}(2\gamma)\,d\gamma\,dx.
\end{equation}
Let $G(\gamma):=\bigl(\widehat{1_{[x,x+2H]}}(\gamma/2)\bigr)^{2}$. Then
\[
\int_{\mathbb{R}}
\Bigl(\widehat{1_{[x,x+2H]}}(-\gamma)\Bigr)^{2}\,
\widehat{1_{[x,x+2H]}}(2\gamma)\,d\gamma
=\frac{1}{2}\int_{\mathbb{R}}\,G (\gamma)  \widehat{1_{[x,x+2H]}}(-\gamma) d\gamma .
\]
Since $$\int_{\mathbb{R}}\,G (\gamma)  \widehat{1_{[x,x+2H]}}(-\gamma)d \gamma = \left(G \ast  \widehat{1_{[x,x+2H]}}\right)(0), $$
using the inverse Fourier transform for the convolution, we have 
$$\left(G \ast  \widehat{1_{[x,x+2H]}}\right)(0)= \int_{\mathbb{R}} \widehat{G}(\alpha) 1_{[x,x+2H]}(\alpha) d\alpha.$$
Note that $$\widehat{G}(\alpha) =2\left(1_{[x,x+2H]} (\gamma) \ast 1_{[x,x+2H]}( \gamma) \right)(2\alpha).$$
Since $\alpha$ is supported in $[x,x+2H]$,  
the first integral in \eqref{Poissonset} is 
\begin{equation}\begin{split}
\int_{X}^{2X} & \int_{x}^{x+2H}\int_{\mathbb{R}}
\,1_{[\max(2\alpha-x-2H,x), \min(2\alpha-x,x+2H)]}(\gamma)\,d\gamma\,d\alpha\,dx
\\&=
\int_{X}^{2X}\int_{x}^{x+H} (2\alpha - 2x)\,d\alpha\,dx + \int_{X}^{2X}\int_{x+H}^{x+2H}  (2x +4H -2 \alpha) d\alpha\,dx
\\&=2H^{2}X.
\end{split}\end{equation}
which evaluates to $2H^{2}X$. Therefore, the main term is
\[
\frac{X}{H}\sum_{q<Q}\ \sum_{\substack{a \bmod q\\(a,q)=1}} C_q^{3}\,H^{2}.
\]

\subsection{Minor arcs}\label{minor}
The contribution from the minor arcs is bounded by
\[
H^{-1}\sup_{\substack{X<x\le 2X\\ \alpha\in\mathfrak{m}}}\bigl|S_f(-2\alpha;x)\bigr|
\int_{X}^{2X}\int_{0}^{1}\bigl|S_f(\alpha;x)\bigr|^{2}\,d\alpha\,dx
\ll
\sup_{\substack{X<x\le 2X\\ \alpha\in\mathfrak{m}}}\bigl|S_f(-2\alpha;x)\bigr|\cdot X(\log X)^{k^{2}}.
\]
Therefore, by Lemma \ref{Lemma13}, the proof is complete.
\section{Proof of Theorem \ref{Theorem2}}\label{Sec4}
By  \eqref{uniformb},
\begin{equation}\label{majorTheorem2}
\begin{split}
\frac{1}{2H} \int_{X}^{2X}\int_{[0,1]} &S_{f_{1}}(\alpha;x)S_{f_{3}}(\alpha;x)\,S_{f_{2}}(-2\alpha;x)\,d\alpha\,dx
\\&\ll
\frac{1}{H} \,\int_{X}^{2X} \sup_{\alpha \in [0,1]}|S_{f_{1}}(\alpha;x)|\
\int_{0}^{1}|S_{f_{2}}(-2\alpha;x)S_{f_{3}}(\alpha;x)| d\alpha dx
\\&\ll XH^{1-\varepsilon}.
\end{split}
\end{equation}

\section*{Acknowledgments}
The author would like to thank Kunjakanan Nath for helpful discussions. 
\medskip

\bibliographystyle{plain}

\end{document}